\documentclass{proc-l}

\usepackage{graphicx}
\usepackage{amsmath}
\usepackage{amscd}
\usepackage{amsfonts}
\usepackage{amssymb}
\usepackage{amsthm}
\usepackage[dvipsnames]{xcolor}
\usepackage{mathtools}
\usepackage{todonotes}
\usepackage{tikz-cd}

\usepackage[hypertexnames=false, colorlinks, citecolor=red, linkcolor=blue, urlcolor=red, bookmarksopen]{hyperref}

\usepackage{enumitem}
\newcounter{dummy}
\makeatletter
\newcommand\myitem[1][]{\item[#1]\refstepcounter{dummy}\def\@currentlabel{#1}}
\makeatother

\usepackage[backend=biber, style=numeric, sorting=anyt, nohashothers=true, maxbibnames=99]{biblatex}
\usepackage[depth=3]{bookmark}
\let\oldtocsection=\tocsection
\let\oldtocsubsection=\tocsubsection
\let\oldtocsubsubsection=\tocsubsubsection
\renewcommand{\tocsection}[2]{\hspace{0em}\vspace{0.5mm}\oldtocsection{#1}{#2}\vspace{0.5mm}}
\renewcommand{\tocsubsection}[2]{\hspace{2em}\vspace{0.25mm}\oldtocsubsection{#1}{#2}\vspace{0.25mm}}
\renewcommand{\tocsubsubsection}[2]{\hspace{2em}\oldtocsubsubsection{#1}{#2}}

\numberwithin{equation}{section}

\newtheorem{theorem}{Theorem}[section]
\newtheorem{lemma}[theorem]{Lemma}
\newtheorem{proposition}[theorem]{Proposition}

\newtheorem{corollary}[theorem]{Corollary}

\theoremstyle{definition}
\newtheorem{example}[theorem]{Example}
\newtheorem{remark}[theorem]{Remark}
\newtheorem{definition}[theorem]{Definition}
\newtheorem{problem}[theorem]{Problem}

\newcommand{\be}{\begin{equation}}
\newcommand{\ee}{\end{equation}}
\newcommand{\bes}{\begin{equation*}}
\newcommand{\ees}{\end{equation*}}

\newcommand{\cC}{\mathcal{C}}
\newcommand{\cD}{\mathcal{D}}

\newcommand{\cF}{\mathcal{F}}

\newcommand{\cH}{\mathcal{H}}

\newcommand{\cL}{\mathcal{L}}

\newcommand{\cP}{\mathcal{P}}
\newcommand{\cQ}{\mathcal{Q}}

\newcommand{\cV}{\mathcal{V}}
\newcommand{\cW}{\mathcal{W}}
\newcommand{\cX}{\mathcal{X}}
\newcommand{\cY}{\mathcal{Y}}

\newcommand{\bB}{\mathbb{B}}
\newcommand{\bC}{\mathbb{C}}
\newcommand{\bD}{\mathbb{D}}

\newcommand{\bN}{\mathbb{N}}
\newcommand{\bQ}{\mathbb{Q}}
\newcommand{\bR}{\mathbb{R}}
\newcommand{\bS}{\mathbb{S}}
\newcommand{\bT}{\mathbb{T}}
\newcommand{\bZ}{\mathbb{Z}}

\newcommand{\dist}{\operatorname{dist}}

\newcommand{\spn}{\operatorname{span}}

\newcommand{\fg}{{\mathfrak{g}}}

\newcommand{\fI}{{\mathfrak{I}}}

\newcommand{\foral}{\text{ for all }}

\newcommand{\AND}{\text{ and }}

\begin{document}


\title{Jointly cyclic polynomials and maximal domains}

\author{
Mikhail Mironov
}
\address{
\hspace{-5mm} Department of Mathematics \\
Technion - Israel Institute of Technology \\
Haifa, Israel
}
\email{
mikhailm@campus.technion.ac.il
}

\author{
Jeet Sampat
}
\address{
\hspace{-5mm} Department of Mathematics \\
Technion - Israel Institute of Technology \\
Haifa, Israel
}
\email{
sampatjeet@campus.technion.ac.il
}


\begin{abstract}
    For a (not necessarily locally convex) topological vector space \texorpdfstring{$\cX$}{X} of holomorphic functions in one complex variable, we show that the shift invariant subspace generated by a set of polynomials is \texorpdfstring{$\cX$}{X} if and only if their common vanishing set contains no point at which the evaluation functional is continuous. For two variables, we show that this problem can be reduced to determining the cyclicity of a single polynomial and obtain partial results for more than two variables. We proceed to examine the maximal domain, i.e., the set of all points for which the evaluation functional is continuous. When \texorpdfstring{$\cX$}{X} is metrizable, we show that the maximal domain must be an \texorpdfstring{$F_\sigma$}{F-sigma} set, and then construct Hilbert function spaces on the unit disk whose maximal domain is the disk plus an arbitrary subset of the boundary that is both \texorpdfstring{$F_\sigma$}{F-sigma} and \texorpdfstring{$G_\delta$}{G-delta}.
    
    \medskip
    
\end{abstract}

\maketitle

\section{Introduction}\label{sec:intro}

\subsection{Cyclicity}\label{subsec:cyclicity}

Let $T$ be a bounded linear self map on a Banach space $\cX$. A vector $v \in \cX$ is said to be \emph{$T$-cyclic} if its \emph{$T$-invariant subspace} is $\cX$, i.e,
\begin{equation*}
    T[v] := \overline{\spn} \left\{ v, Tv, T^2 v, \dots \right\} = \cX.
\end{equation*}
Here, $\overline{\spn}$ denotes the closed linear hull in $\cX$. For general infinite dimensional spaces $\cX$, the problem of characterizing $T$-cyclic vectors is quite tricky.

Consider the simplest infinite dimensional space, which is $\ell^2(\bC)$ -- the Hilbert space of all square-summable complex sequences -- along with the \emph{shift operator}
\begin{equation*}
    S : (a_0,a_1,\dots) \mapsto (0, a_0, a_1, \dots).
\end{equation*}
A classical result of Beurling (see \cite[Theorem I]{Beu49}) shows that in order to characterize the shift cyclic vectors in $\ell^2(\bC)$, we must consider the corresponding function space, i.e., the \emph{Hardy space} $H^2(\bD)$ of holomorphic functions on the unit disk $\bD$ with square-summable power-series coefficients. Beurling showed that $f \in H^2(\bD)$ corresponds to a shift-cyclic vector if and only if it is \emph{outer}, i.e.,
\begin{equation}\label{eqn:outer.func}
    -\infty < \log |f(0)| = \int_0^{2\pi} \log |f(e^{it})| \frac{dt}{2\pi}.
\end{equation}

Such a classification is not known to exist even for function spaces in one complex variable, such as the \emph{Dirichlet space} $\cD$ (see \cite[Question 12]{BS84}) and the \emph{Bergman space} $L^2_a(\bD)$ (see \cite[Section 7.5]{HKZ00}), which we shall introduce below.
While a complete characterization of cyclic functions seems out of reach in general function spaces, some concrete results about the cyclicity of polynomials have been known for quite some time. Let us survey some of these results to bring context into our work.

\subsection{Motivation}\label{subsec:motivation}

Let $\Omega \subset \bC^d$ be an open set for some $d \in \bN$ and denote $\operatorname{Hol(\Omega)}$ to be the collection of all holomorphic functions on $\Omega$. Henceforth, $\cX \subset \operatorname{Hol}(\Omega)$ will represent a Hausdorff topological vector space (TVS) that satisfies:

\begin{enumerate}
    \myitem[\textbf{P1}] The set of \emph{polynomials $\cP_d$ in $d$ complex variables} is dense in $\cX$. \label{item:P1}
    \myitem[\textbf{P2}] The \emph{evaluation functional} $\Lambda_w$ at each $w \in \Omega$ is continuous, where \label{item:P2}
    \begin{equation*}
        \Lambda_w(F) := F(w) \foral F \in \cX.
    \end{equation*}
    \myitem[\textbf{P3}] The \emph{shift operator} $S_j$ for each $1 \leq j \leq d$ is continuous, where \label{item:P3}
    \begin{equation*}
        S_j F(z) := z_j F(z) \foral F \in \cX.
    \end{equation*}
\end{enumerate}

We write $\cX^*$ for the space of continuous linear functionals on $\cX$ and $\cL(\cX)$ for the space of continuous linear self maps on $\cX$. Note that the following notion of cyclicity generalizes that of shift-cyclicity in the one variable case.

\begin{definition}\label{def:shift-cyc.func}
    Let $\cX$ be as above.
    \begin{enumerate}
        \item For $F \in \cX$, we define its \emph{invariant subspace} as
        \begin{equation*}
            S[F] := \overline{F \cP_d} = \overline{\left\{PF : P \in \cP_d\right\}}.
        \end{equation*}
        \item $F \in \cX$ is said to be \emph{cyclic} if $S[F] = \cX$.
        \item A family $\cF \subset \cX$ is said to be \emph{jointly cyclic} if $S[\cF] = \cX$, where
\begin{equation*}
    S[\cF] := \overline{\spn}_{F \in \cF} S[F].
\end{equation*}
    \end{enumerate}
\end{definition}

Let us mention a few examples of Hausdorff TVSs of holomorphic functions.

\subsubsection{\texorpdfstring{\textbf{Hardy spaces}}{Hardy spaces}}

Let $\bD^d := \bD \times \dots \times \bD$ be the \emph{unit polydisk} and $\bT^d := \bT \times \dots \times \bT$ be the \emph{unit $d$-torus} in $\bC^d$, where $\bT$ is the \emph{unit circle} in $\bC$. Let $\sigma_d$ be the \emph{normalized Lebesgue measure on $\bT^d$} and define the \emph{Hardy spaces on $\bD^d$} for any $0 < p < \infty$ as
\begin{equation*}
    H^p(\bD^d) := \left \{ F \in \operatorname{Hol}(\bD^d) : \sup_{r < 1} \int_{\bT^d} |F(rw)|^p d\sigma_d(w) < \infty \right \}.
\end{equation*}
$H^p(\bD^d)$ turns into a Banach space for $1 \leq p < \infty$ and an $F$-space for $0 < p < 1$, and satisfies \ref{item:P1}-\ref{item:P3} for all $0 < p < \infty$. For $d = 1$, cyclic functions are characterized by \eqref{eqn:outer.func} as well (see \cite[Theorem 7.4]{Dur70} for $1 \leq p < \infty$ and \cite[Theorem 4]{Gam66} for $0 < p < 1$). For $d > 1$, one can generalize \eqref{eqn:outer.func} in a natural way, but it is not equivalent to cyclicity (see \cite[Theorem 4.4.8]{Rud69}).
\cite[Theorem 5]{NGN70} shows that a polynomial is cyclic in $H^p(\bD^d)$ if and only if it is non-vanishing on $\bD^d$. We shall see a general theme that the zero set of a polynomial is closely linked to its cyclicity.

\subsubsection{\texorpdfstring{\textbf{Dirichlet-type spaces}}{Dirichlet-type spaces}}

Every $F \in \operatorname{Hol}(\bD^d)$ has a power-series representation
\begin{equation*}
    F(z) = \sum_{\alpha \in \bZ_+^d} c_\alpha z^\alpha.
\end{equation*}
Here, $\alpha = (\alpha_1,\dots,\alpha_d)$ is a $d$-tuple of non-negative integers $\bZ_+$ and $z^\alpha := z_1^{\alpha_1}\dots z_d^{\alpha_d}$. For any $t \in \bR$, we define the \emph{Dirichlet-type spaces on $\bD^d$} as
\begin{equation*}
    \cD_t(\bD^d) := \left\{ F(z) = \sum_{\alpha \in \bZ_+^d} c_\alpha z^\alpha : \sum_{\alpha \in \bZ_+^d} \left(\prod_{j=1}^d (\alpha_j + 1)^t\right) |c_\alpha|^2 < \infty \right\}.
\end{equation*}
It is easy to check that $\cD_t(\bD^d)$ is a Hilbert space that satisfies \ref{item:P1}-\ref{item:P3} over $\bD^d$. For $t = 0$ we recover the Hardy space $H^2(\bD^d)$, for $t = 1$ we get the usual \emph{Dirichlet space} $\cD(\bD^d)$, and for $t = -1$ we get the \emph{Bergman space} $L^2_a(\bD^d)$.

Cyclic functions in $\cD_t(\bD^d)$ remain mysterious even when $d = 1$, but for $d = 1, 2$ we know what the cyclic polynomials are. The main result of \cite{BKKLSS16} shows that $P \in \cP_2$ is cyclic in $\cD_t(\bD^2)$ if and only if
\begin{enumerate}
    \item $P$ is non-vanishing on $\bD^2$ when $t \leq 1/2$,
    \item $P$ is non-vanishing on $\bD^2$ and has at most finitely many zeroes in $\bT^2$, or $P$ is a multiple of $z_1 - w$ or of $z_2 - w$ for some $w \in \bT$ when $1/2 < t \leq 1$.
    \item $P$ is non-vanishing on $\overline{\bD^2}$ when $t > 1$.
\end{enumerate}
The problem of determining cyclic polynomials in $\cD_t(\bD^d)$ for $d > 2$ is still open.

For a similar class of spaces $\cD_t(\bB_d)$ over the \emph{Euclidean unit ball} $\bB_d \subset \bC^d$, cyclic polynomials are determined for $d=2$ in \cite[Theorem 3]{KV23}, but the problem remains open for $d > 2$ (see \cite{VZ24}). Similar results about cyclicity for other spaces in one and several variables can be found in \cite{APRSS24, BH97, CG03, CMR20, EKMR14, KLRS19, Nik12, ST21}.

\subsection{Main results}\label{subsec:main.results}


Let $\cX$ be a Hausdorff TVS satisfying \ref{item:P1}-\ref{item:P3} on an open set $\Omega \subset \bC^d$. Following \cite{Sam21}, we introduce the \emph{maximal domain} $\Omega_{max}$, which consists of all $w \in \bC^d$ such that the evaluation functional $\Lambda_w : P \mapsto P(w)$ defined on polynomials extends to an element of $\cX^*$. We also consider the \emph{enveloping domain of cyclic polynomials} $\Omega_{env}$, which consists of all $w \in \bC^d$ such that no cyclic polynomial vanishes at $w$.

\subsubsection{\texorpdfstring{\textbf{Section \ref{sec:dom.assoc.with.cyc}}}{Section 2}}

In Theorem \ref{thm:max.dom.is.point.spec.of.adjoint}, we show that $\Omega_{max} = \sigma_p(S^*)$, where $\sigma_p(S^*)$ is the set of all \emph{joint eigenvalues} of the topological adjoints of the shift operators. In Proposition \ref{prop:max.dom.in.env.cyc.poly}, we show that $\Omega_{max} \subseteq \Omega_{env}$, which happens to be an equality for the spaces introduced in Section \ref{subsec:motivation} (see Examples \ref{example:find.max.dom.1} and \ref{example:find.max.dom.2}).
In later sections, we tackle the question of determining when equality above holds, but are left with an open problem (see Problem \ref{prob:max.dom.equal.env.cyc.poly}).

\subsubsection{\texorpdfstring{\textbf{Section \ref{sec:joint.cyc.poly}}}{Section 3}} In Theorem \ref{thm:description.max.dom}, we show that $\Omega_{max}$ is precisely the collection of all $w \in \bC^d$ for which the family
\begin{equation*}
    \cF_w := \left\{ z_j - w_j : 1 \leq j \leq d \right\}
\end{equation*}
is not jointly cyclic. For $d = 1$, we combine this theorem with basic facts from commutative algebra to obtain Corollary \ref{cor:joint.cyc.poly.1-var}, which states that $\cF \subset \cP_1$ is jointly cyclic if and only if its common vanishing set $V(\cF)$ is disjoint from $\Omega_{max}$. In Corollary \ref{cor:joint.cyc.in.2-var}, we show that $\cF \subset \cP_2$ is jointly cyclic if and only if the G.C.D. of $\cF$ is cyclic and $V(\cF) \cap \Omega_{max} = \emptyset$. In particular, we obtain a classification of jointly cyclic polynomials in $H^p(\bD^2)$, $\cD_t(\bD^2)$ and $\cD_t(\bB_2)$.
This result follows from Theorem \ref{thm:joint.cyc.poly}, which also provides a partial result for $d > 2$. These results generalize previously known results for various Banach function spaces for which $\Omega_{max} = \Omega$ or $\overline{\Omega}$ (cf. \cite[Chapter 2]{CG03}, \cite[Section 10.1]{CMR20} and \cite[Sections 9.1 and 9.2]{EKMR14}).

\subsubsection{\texorpdfstring{\textbf{Section \ref{sec:classify.max.dom}}}{Section 4}}

Given the relationship between $\Omega_{max}$ and cyclic polynomials, it is natural to ask whether we could identify subsets of $\bC^d$ that appear as $\Omega_{max}$ for some choice of $\cX$.
In Theorem \ref{thm: maximal domain is F-sigma}, we show that if $\cX$ is metrizable then $\Omega_{max}$ is always an $F_\sigma$ set (countable union of closed sets). In Theorem \ref{thm: Fs and Gd Gamma}, we construct Hilbert function spaces $\cH$ that satisfy \ref{item:P1}-\ref{item:P3} over $\bD$ and for which $\Omega_{max} = \bD \cup \Gamma$, where $\Gamma \subseteq \bT$ can be chosen to be any set that is both $F_\sigma$ and $G_\delta$ (countable intersection of open sets). We
leave open the problem to determine if $\Gamma$ above can be chosen to be any $F_\sigma$ set (see Problem \ref{prob:dichotomy.max.dom}). For locally convex TVSs, we show that this can be achieved in Remark \ref{rem:loc.con.sp.with.F_sigma.max.dom}.

\section{Domains associated with cyclicity}\label{sec:dom.assoc.with.cyc}


\subsection{Maximal domain}

Let $\cX$ be a Hausdorff TVS of holomorphic functions on an open set $\Omega \subset \bC^d$ satisfying \ref{item:P1}-\ref{item:P3}. Define the maximal domain of $\cX$ as
\begin{equation*}
    \Omega_{max} := \left\{ w \in \bC^d : \Lambda_w(P) := P(w) \foral P \in \cP_d \text{ extends to } \Lambda_w \in \cX^* \right\}.
\end{equation*}
It was established in \cite{Sam21} that -- at least when $\cX$ is a Banach space -- $\Omega_{max}$ sits inside the \emph{right Harte spectrum} $\sigma_R(S)$ of the $d$-tuple of the shift operators. However, we now understand that the maximal domain is precisely the set of \emph{joint eigenvalues} of the $d$-tuple of the topological adjoints\footnote{The topological adjoint $T^* \in \cL(\cX^*)$ of $T$ is given by $T^*\Lambda(x) = \Lambda(Tx) \foral x \in \cX, \, \Lambda \in \cX^*$.} of $S$. Recall that for a $d$-tuple of operators $T = (T_1,\dots,T_d)$ on some space $\cY$, the set of its joint eigenvalues is defined as
\begin{equation*}
    \sigma_p(T) := \left\{ w \in \bC^d : T_j h = w_j h, \, \forall 1 \leq j \leq d \text{ for some } 0 \neq h \in \cY \right\}
\end{equation*}

\begin{theorem}\label{thm:max.dom.is.point.spec.of.adjoint}
    If $S^* := (S_1^*,\dots,S_d^*)$, then
    \begin{equation}\label{eqn:max.dom.is.point.spec.of.adjoint}
        \Omega \subseteq \Omega_{max} = \sigma_p(S^*) \subseteq \sigma_R(S).
    \end{equation}
\end{theorem}

\begin{proof}
    \ref{item:P2} implies that $\Omega \subseteq \Omega_{max}$, and it is easy to check that $\sigma_p(T^*) \subseteq \sigma_R(T)$ always holds for any $T$. Thus, we only need to verify the equality in the middle.
    
    Suppose $w \in \Omega_{max}$. For each $P \in \cP_d$ and $1 \leq j \leq d$, note that
    \begin{equation}\label{eqn:Lambda_w.is.multiplicative}
        S_j^* \Lambda_w (P) = \Lambda_w(S_j P) = \Lambda_w(z_j P) = w_j P(w) = w_j \Lambda_w(P).
    \end{equation}
    It follows from \ref{item:P1} that $w \in \sigma_p(S^*)$ with joint eigenvector $\Lambda_w$.

    Conversely, suppose $w \in \sigma_p(S^*)$ and let $0 \neq \Lambda \in \cX^*$ be a joint eigenvector for $w$. Note that for each $P \in \cP_d$ we have
    \begin{equation*}
        \Lambda(P) = \Lambda(P(S) 1) = (P(S^*) \Lambda)1 = P(w)\Lambda(1).
    \end{equation*}
    Thus, $\Lambda(1) \neq 0$ and $\Lambda_w \vert_{\cP_d} = \frac{1}{\Lambda(1)} \Lambda \vert_{\cP_d}$ extends continuously to $\cX$, as required.
\end{proof}

\begin{remark}\label{rem:top.adjoint.vs.Hilbert.sp.adjoint}
    When $\cX$ is a Hilbert space, the topological and the Hilbert space adjoint of an operator are connected by an anti-linear mapping. So if $S^*$ is the $d$-tuple of the Hilbert space adjoints of $S_j$'s, then
    \begin{equation*}
        \Omega_{max} = \sigma_p(S^*)^* := \left\{ w \in \bC^d : w^* = (\overline{w_1},\dots,\overline{w_d}) \in \sigma_p(S^*) \right\}.
    \end{equation*}
\end{remark}

\begin{example}\label{example:find.max.dom.1}
    Let $\cX$ be a Banach space, and note that one typically obtains $\sigma_R(S) = \overline{\Omega}$. Two common situations arise when computing $\Omega_{max}$ in this case.

    \begin{enumerate}
        \item Let $\cX = \cD_t(\bD^d)$ with $t > 1$ or $\cD_t(\bB_d)$ with $t > d$. It is easy to check that $\cX$ is an algebra, and that
        point-evaluations are bounded even for $w \in \partial \Omega$. Thus, by Theorem \ref{thm:max.dom.is.point.spec.of.adjoint}, it follows that $\Omega_{max} = \overline{\Omega}$.

        \item Let $r \Omega \subset \Omega$ for all $r < 1$. Suppose $F \in \cX$ has the property that the \emph{radial dilations} $F_r : z \mapsto F(rz)$ converge to $F$ in the norm as $r \to 1$, and that each $F_r$ can be approximated by polynomials $P_{r,k}$ uniformly on $\overline{\Omega}$ and also in the norm. If $w \in \Omega_{max} \cap \partial \Omega$, then for such an $F$ we get
        \begin{equation*}
            \Lambda_{w}(F) = \lim_{r \to 1} \Lambda_w(F_r) = \lim_{r \to 1} \lim_{k \to \infty} P_{r,k}(w) = \lim_{r \to 1} F(rw).
        \end{equation*}
        If, for some $w \in \partial \Omega$, we can find $F \in \cX$ as above such that $\lim_{r \to 1} F(rw)$ does not exist, then we can conclude that $w \not\in \Omega_{max}$.

        For every $F \in H^p(\bD^d)$, both approximation properties hold (even for $p < 1$) (see \cite[Section 3.4]{Rud69}). Moreover, given any $w \in \overline{\bD^d}$ with $w_j \in \bT$ (say), we choose $F(z) = \log (z_j - w_j)$ to be as above. Thus, $\Omega_{max} = \bD^d$.
    \end{enumerate}
\end{example}

The argument surrounding \eqref{eqn:Lambda_w.is.multiplicative} above can be easily generalized using \ref{item:P3} to obtain a partial multiplicative property of the functionals $\Lambda_w$ where $w \in \Omega_{max}$. We leave the easy proof of this fact and its corollary for the reader to verify.

\begin{proposition}\label{prop.point-evals.are.multiplicative}
    For each $P \in \cP_d$ and $F \in \cX$, we have
    \begin{equation*}
        \Lambda_w(PF) = P(w) \Lambda_w(F) \foral w \in \Omega_{max}.
    \end{equation*}
\end{proposition}


\begin{corollary}\label{cor:cyc.necessary.condition}
    If $F \in \cX$ is cyclic, then
    \begin{equation*}
        \Lambda_w(F) \neq 0 \foral w \in \Omega_{max}.
    \end{equation*}
\end{corollary}

\begin{remark}\label{remark:algebraic.consistency}
    The above partial multiplicative property of $\Lambda_w$ in Proposition \ref{prop.point-evals.are.multiplicative} is why the name `maximal domain' is appropriate for $\Omega_{max}$. This is because different notions of maximal domains have been studied using different kinds of partially multiplicative functionals (cf. \cite[Section 5]{Har17} and \cite[Section 2]{MS17}). However, if $\cX$ is a Banach space then all of these notions are equivalent (see \cite[Theorem 1.7]{Sam21}).
\end{remark}

\subsection{Enveloping domain of cyclic polynomials}\label{subsec:env.cyc.poly}

Let $\cC_\cP$ be the class of cyclic polynomials in $\cX$ and define the enveloping domain of cyclic polynomials as
\begin{equation*}
    \Omega_{env} := \left\{ w \in \bC^d : P(w) \neq 0 \foral P \in \cC_\cP \right\}.
\end{equation*}

\begin{proposition}\label{prop:max.dom.in.env.cyc.poly}
    $\Omega_{max} \subseteq \Omega_{env}$.
\end{proposition}

\begin{proof}
    Using Corollary \ref{cor:cyc.necessary.condition}, if $w \in \Omega_{max}$ then $P(w) \neq 0$ for any $P \in \cC_\cP$.
\end{proof}

One can typically identify $\Omega_{env}$ using a special class of cyclic polynomials.

\begin{example}\label{example:find.max.dom.2}
    For $H^p(\bD^d)$ with $0 < p < \infty$ and $\cD_t(\bD^d)$ with $t \leq 1$, we know that $z_j - w$ is cyclic if and only if $w \in \bC \setminus \bD$. It follows that $\Omega_{env} \subseteq \bD^d$ and thus, $\Omega_{max} = \Omega_{env} = \bD^d$ in this case. The same family shows that for $\cD_t(\bB_d)$ with $t \leq d$, we have $\Omega_{max} = \Omega_{env} = \bB_d$ (use \cite[Theorem 11 and Lemma 15]{Vav23}).
    
    For $\cD_t(\bD^d)$ with $t > 1$ or $\cD_t(\bB_d)$ with $t > d$, $z_j - w$ is cyclic if and only if $w \in \bC \setminus \overline{\bD}$. Thus, using Example \ref{example:find.max.dom.1} $(1)$, we have $\Omega_{max} = \Omega_{env} = \overline{\bD^d}$ in this case.
\end{example}

This leads us to a question that we are unable to answer for an arbitrary $\cX$. In Section \ref{sec:joint.cyc.poly}, we show that the answer is always affirmative when $d = 1$.

\begin{problem}\label{prob:max.dom.equal.env.cyc.poly}
    If $\cX$ is a Banach/Hilbert space, is it always true that $\Omega_{max} = \Omega_{env}$?
\end{problem}

We end this section with an algebraic property of cyclic polynomials that is well known for Banach spaces. We provide a short proof in our general setup.

\begin{proposition}\label{prop:alg.prop.of.cyc.poly}
    If $P \in \cP_d, F \in \cX$, then $PF$ is cyclic if and only if $P, F$ are both cyclic. Consequently, $P \in \cC_\cP$ if and only if its irreducible factors are cyclic.
\end{proposition}

\begin{proof}
    By \ref{item:P3}, multiplication by $P \in \cP_d$ is a continuous linear map on $\cX$. Thus,
    \begin{equation}\label{eqn:cont.of.mult.by.poly}
        \overline{P \overline{\cY}} = \overline{P \cY} \foral \cY \subset \cX.
    \end{equation}
    
    If $P$ and $F$ are both cyclic, then
    \begin{equation*}
        S[PF] = \overline{PF \cP_d} = \overline{P \overline{F \cP_d}} = \overline{P \cX} \supseteq S[P] = \cX.
    \end{equation*}
    It follows that $PF$ is cyclic. Conversely, suppose $PF$ is cyclic.
    Clearly, $PF \in S[F]$ and we get that $S[PF] \subseteq S[F]$, so $F$ is cyclic. We also have
    \begin{equation*}
        PF \in P \overline{\cP_d} \subseteq \overline{P \cP_d} \subseteq S[P]
    \end{equation*}
    and thus, $S[PF] \subseteq S[P]$. This completes the proof.
\end{proof}

\begin{remark}\label{rem:env.in.Tay.spec}
    When $\cX$ is a Banach space, we can use the \emph{Taylor functional calculus} to show that $\Omega_{env}$ lies in the convex hull $\widehat{\Omega}$ of the \emph{Taylor joint spectrum} $\sigma_{Tay}(S)$ of $S$ (see \cite[Section 14.4]{AM02}). To see this, let $w \in \bC^d \setminus \widehat{\Omega}$ and use the Hahn-Banach separation theorem to find an affine map $P$ that vanishes at $w$ but is non-vanishing on $\widehat{\Omega}$. Then, $1/P \in \cX$ by the Taylor functional calculus and hence, applying Proposition \ref{prop:alg.prop.of.cyc.poly} with $F = 1/P$, we get that $P \in \cC_\cP$. Since $P$ was chosen so that $P(w) = 0$, it follows that $w \in \bC^d \setminus \Omega_{env}$.
\end{remark}

\section{Joint cyclicity of polynomials}\label{sec:joint.cyc.poly}

We now demonstrate the connection between maximal domains and cyclicity. Let us first showcase the one variable case to ease our way into the discussion.

\subsection{Cyclic polynomials in one variable}\label{subsec:cyc.poly.1-var}

The main result of this subsection follows from an alternate description of $\Omega_{max}$ in general. We introduce the family
    \begin{equation}\label{eqn:F_w}
        \cF_w := \left\{ z_j - w_j : 1 \leq j \leq d \right\} \foral w \in \bC^d.
    \end{equation}

\begin{theorem}\label{thm:description.max.dom}
    For all $w \in \bC^d$, we have
    \begin{equation}\label{eqn:X=1+inv.subspace}
        \cX = \bC 1 + S[\cF_w].
    \end{equation}
    Moreover, the following are equivalent:
    \begin{enumerate}
        \item $w \in \Omega_{max}$.
        \item $\cF_w$ is not jointly cyclic.
        \item $\cX = \mathbb{C}1 \oplus S[\cF_w]$ as an algebraic direct sum.
    \end{enumerate}
\end{theorem}

\begin{proof}
    \eqref{eqn:X=1+inv.subspace} follows at once by noting that its R.H.S. is a closed subspace of $\cX$ that contains $\cP_d$. This is because, for every $P \in \cP_d$ and $w \in \bC^d$, we can find $\{Q_1,\dots,Q_d\} \subset \cP_d$ such that
    \begin{equation}\label{eqn:P(z)=P(w)+blah}
        P(z) = P(w) + (z_1 - w_1)Q_1(z) + \dots + (z_d - w_d)Q_d(z).
    \end{equation}
    $(2) \Leftrightarrow (3)$ holds since the R.H.S. of \eqref{eqn:X=1+inv.subspace} is an algebraic direct sum if and only if
    \begin{equation*}
        \bC 1 \cap S[\cF_w] = \{ 0 \}.
    \end{equation*}
    $(1) \Rightarrow (2)$ holds by noting that if $\Lambda_w$ is continuous, then $\Lambda_w \equiv 0$ on $S[\cF_w]$ and thus, $1 \notin S[\cF_w]$. $(3) \Rightarrow (1)$ follows since the algebraic projection on $\bC 1$ satisfies
    \begin{equation*}
        \operatorname{Proj_{\bC 1}}(P) = \Lambda_w(P) \foral P \in \cP_d
    \end{equation*}
    using \eqref{eqn:P(z)=P(w)+blah}, and is a non-zero linear functional on $\cX$ whose kernel is closed, which is automatically continuous (see \cite[Theorem 1.18]{Rud73}). Here, we identify $\bC 1 \cong \bC$ to view $\operatorname{Proj}_{\bC 1}$ as a linear functional on $\cX$. This completes the proof.
\end{proof}

The main result now follows at once from Proposition \ref{prop:alg.prop.of.cyc.poly} and Theorem \ref{thm:description.max.dom}.

\begin{corollary}\label{cor:cyc.poly.in.1-var}
    If $\cX$ is a Hausdorff TVS of holomorphic functions on an open set $\Omega \subset \bC$ satisfying \ref{item:P1}-\ref{item:P3}, then
    \begin{equation*}
        \Omega_{max} = \left\{ w \in \bC : z- w \text{ is not cyclic} \right\} = \Omega_{env}.
    \end{equation*}

    Consequently, $P \in \cC_\cP$ if and only if $P(w) \neq 0 \foral w \in \Omega_{max}$.
\end{corollary}

Corollary \ref{cor:cyc.poly.in.1-var} generalizes several known results of this type (e.g. \cite[Theorem 9.2.1]{CMR20} and \cite[Lemma 10.1.3]{EKMR14}). Also, note that the final assertion of Corollary \ref{cor:cyc.poly.in.1-var} does not hold when $d > 1$ as illustrated by \cite[Example 2]{BCLSS15}, since $1 - z_1 z_2 \in \cD_t(\bD^2)$ is not cyclic for $1/2 < t \leq 1$ despite being non-vanishing on $\Omega_{max} = \bD^d$.

\subsection{Polynomial ideals with finite codimension}\label{subsec:poly.ideal.fin.codim}

Let us view $\cP_d$ as the ring of polynomials in $d$-variables over $\bC$. For a family $\cF \subset \cP_d$, let $V(\cF)$ denote the \emph{affine algebraic variety associated to $\cF$}, i.e.,
\begin{equation*}
V(\cF) := \left\{ w \in \bC^d : P(w) = 0 \foral P \in \cF \right\}.
\end{equation*}
We also write $I(\cF)$ for the \emph{ideal generated by $\cF$}, and note that $V(I(\cF)) = V(\cF)$ for any $\cF \subset \cP_d$. Hilbert's Nullstellensatz tells us that ideals of the form
\begin{equation}\label{eqn:max.ideal.of.P_d}
    I(\cF_w) = \langle z_j - w_j : 1 \leq j \leq d \rangle,
\end{equation}
where $w \in \bC^d$ and $\cF_w$ is as in \eqref{eqn:F_w}, are all the maximal ideals of $\cP_d$.

Ahern and Clark \cite{AC70} showcased the connection between finite codimensional ideals (as $\bC$-vector spaces) and finite codimensional invariant subspaces of $H^2(\bD^d)$. Later works (cf. \cite{DP89, Guo99}) generalized their result to \emph{analytic Hilbert modules}, i.e., Hilbert spaces $\cH$ satisfying \ref{item:P1}-\ref{item:P3} over an open set $\Omega \subset \bC^d$ such that $\Omega_{max} = \Omega$ (also see \cite[Chapter 2]{CG03}). In this paper, we focus solely on determining joint cyclicity and not finite codimensionality of invariant subspaces generated by polynomials. For this, we require two important tools from commutative algebra, first of which is a classical result (see \cite[Proposition 3.7.1]{KR00}).

\begin{proposition}\label{prop:identify.ideal.fin.codim}
    An ideal $\fI$ has finite codimension if and only if $V(\fI)$ is finite.
\end{proposition}

The reader is referred to \cite[Section 3.7]{KR00} for more details on how to algorithmically determine whether a family of polynomials has finitely many common zeroes. Let us now see how we can algebraically reduce the problem of joint cyclicity. Since $\cP_d$ is a unique factorization domain, we can define the \emph{greatest common divisor} $\fg(\cF)$ of any family $\cF \subset \cP_d$. This allows us to write
\begin{equation}\label{eqn:fac.form.ideal}
    I(\cF) = \fg(\cF) \fI_\cF \foral \cF \subset \cP_d.
\end{equation}

Note that $\fg(\fI_\cF) = 1$ is always true. For $d = 1$, we have $\fI_\cF = \cP_1$, since $\cP_1$ is a principal ideal domain, and it follows that $\fg(\cF)$ determines $I(\cF)$ completely. For $d = 2$, we have the following lemma due to Yang (see \cite[Lemma 6.1]{Yan99}).

\begin{lemma}\label{lemma:fin.codim.ideal.in.P_2}
    For each $\cF \subset \cP_2$, the ideal $\fI_\cF$ as in \eqref{eqn:fac.form.ideal} has finite codimension.
\end{lemma}

Clearly, this lemma does not hold for $d > 2$; simply take $\cF = \{ z_1, z_2 \} \subset \cP_3$ and note that $\fg(\cF) = 1$ but $I(\cF)$ has infinite codimension.

\subsection{Joint cyclicity of polynomials}\label{subsec:joint.cyc.poly}

Let us start with an easy generalization of Proposition \ref{prop:alg.prop.of.cyc.poly}. Given two families $\cQ \subset \cP_d$ and $\cF \subset \cX$, we define
    \begin{equation*}
        \cQ \cF := \{ P F : P \in \cQ \AND F \in \cF \} \subset \cX.
    \end{equation*}

\begin{proposition}\label{prop:alg.prop.joint.cyc.poly}
    If $\cQ \subset \cP_d$ and $\cF \subset \cX$ are two families, then $\cQ \cF$ is jointly cyclic if and only if $\cQ$, $\cF$ are both jointly cyclic. In particular, $\cQ$ is jointly cyclic if and only if $\fg(\cQ)$ is cyclic and $\fI_\cQ$ is jointly cyclic.
\end{proposition}

\begin{proof}
    If $P \in \cQ$ and $F \in \cF$, then, as in the proof of Proposition \ref{prop:alg.prop.of.cyc.poly}, we get
    \begin{equation*}
        PF \in S[P] \cap S[F] \subset S[\cQ] \cap S[\cF].
    \end{equation*}
    It follows that $S[\cQ \cF] \subseteq S[\cQ] \cap S[\cF]$, and thus $\cQ, \cF$ are both jointly cyclic.

    Conversely, suppose $\cQ,\cF$ are both jointly cyclic. Using \eqref{eqn:cont.of.mult.by.poly}, we note that
    \begin{equation*}
        S[\cQ \cF] = \overline{\spn}_{\substack{P \in \cQ \\ F \in \cF}} \overline{PF \cP_d} =  \overline{\spn}_{P \in \cQ}\overline{ P S[\cF]} = \overline{\spn}_{P \in \cQ} \overline{P \cX} = S[\cQ] = \cX.
    \end{equation*}
    Thus, $\cQ \cF$ is jointly cyclic. The final assertion now follows from the first part.
\end{proof}

As noted after \eqref{eqn:fac.form.ideal}, we know that $\fI_\cF = \cP_1$ for $d = 1$ and thus, we can combine Corollary \ref{cor:cyc.poly.in.1-var} and Proposition \ref{prop:alg.prop.joint.cyc.poly} to easily obtain the following result.

\begin{corollary}\label{cor:joint.cyc.poly.1-var}
    If $\cX$ is a Hausdorff TVS of holomorphic functions satisfying \ref{item:P1}-\ref{item:P3} on an open set $\Omega \subset \bC$, then a family of polynomials $\cF \subset \cP_1$ is jointly cyclic if and only if $V(\cF) \cap \Omega_{max} = \emptyset$.
\end{corollary}

Let us now state the main result of this section. Note that the theorem below is a direct generalization of the equivalence $(1) \Leftrightarrow (2)$ in Theorem \ref{thm:description.max.dom}, since the families $\cF_w$ are such that $\fg(\cF_w) = 1$ and $\fI_{\cF_w} = I(\cF_w)$ has codimension $1$.

\begin{theorem}\label{thm:joint.cyc.poly}
    Let $\cF \subset \cP_d$ be a family of polynomials such that $V(\fI_\cF)$ is finite. Then, $\cF$ is jointly cyclic if and only if $V(\cF) \cap \Omega_{max} = \emptyset$ and $\fg(\cF)$ is cyclic.
\end{theorem}

\begin{proof}
    Using Propositions \ref{prop:identify.ideal.fin.codim} and \ref{prop:alg.prop.joint.cyc.poly}, it suffices to show that if $\fI_\cF$ has finite codimension (say $k$), then $\fI_\cF$ is jointly cyclic if and only if $V(\fI_\cF) \cap \Omega_{max} = \emptyset$.

    To this end, suppose $V(\fI_\cF) \cap \Omega_{max} = \emptyset$ and let $\cV = \spn \{Q_1, \dots, Q_k\} \subset \cP_d$ be such that
    \begin{equation*}
        \fI_\cF + \cV = \cP_d
    \end{equation*}
    as a $\bC$ vector space. Without loss of generality, we can assume that there exists $1 \leq l \leq k$ such that $\cW = \spn\{ Q_1,\dots,Q_l \} \subset S[\fI_\cF]$ and
    \begin{equation*}
        S[\fI_\cF] \oplus (\cV \ominus \cW) = \cX
    \end{equation*}
    as an algebraic direct sum. We want to show that $l = k$, i.e., $\cV \ominus \cW = \{0\}$.
    
    First, we claim that $V(\fI_\cF + \cW) = \emptyset$. Indeed, if $w \in V(\fI_\cF + \cW)$, then for each $F = G + P \in S[\fI_\cF] \oplus (\cV \ominus \cW) = \cX$, we can define
    \begin{equation*}
        \Lambda(F) := \Lambda_w \big\vert_{\cV \ominus \cW}(P).
    \end{equation*}
    Now, since
    \begin{equation*}
        (\fI_\cF + \cW) \oplus (\cV \ominus \cW) = \cP_d,
    \end{equation*}
    it follows that $\Lambda \equiv \Lambda_w $ on $\cP_d$ and
    \begin{equation*}
        \ker \Lambda = S[\fI_\cF] \oplus \ker\Lambda_w \big\vert_{\cV \ominus \cW}.
    \end{equation*}
    This implies that $\ker \Lambda$ is closed, and we can invoke \cite[Theorem 1.18]{Rud73} to show that $\Lambda_w \in \cX^*$. Thus, $w \in \Omega_{max}$ as a result, but this cannot happen since
    \begin{equation*}
        w \in V(\fI_\cF + \cW) \cap \Omega_{max} \subseteq V(\fI_\cF) \cap \Omega_{max} = \emptyset.
    \end{equation*}
    This shows that $V(\fI_\cF + \cW) = \emptyset$ and, by Hilbert's Nullstellensatz, we get that $I(\fI_\cF + \cW) = \cP_d$ since it is a nontrivial ideal that cannot be contained in any maximal ideal (see \eqref{eqn:max.ideal.of.P_d}). It then follows that
    \begin{equation*}
        S[\fI_\cF] = S[\fI_\cF + \cW] = \overline{I(\fI_\cF + \cW)} = \overline{\cP_d} = \cX
    \end{equation*}
    and $\cV \ominus \cW = \{0\}$, as required. Thus, $\fI_\cF$ is jointly cyclic.

    The converse is easy to see, since $w \in V(\fI_\cF) \cap \Omega_{max}$ implies that $\Lambda_w \equiv 0$ on $S[\fI_\cF]$. This further implies that $1 \notin S[\fI_\cF]$ and hence, $\fI_\cF$ is not jointly cyclic.
\end{proof}

The assumption that $V(\fI_\cF)$ is finite cannot be dropped in general from the hypothesis of Theorem \ref{thm:joint.cyc.poly} as the following example shows.

\begin{example}\label{example:infinite.variety}
    This example requires several tools and notations used in \cite{Vav23}. To keep this discussion compact, we provide a sketch of the argument. Consider $\cF = \{ 1 - 2 z_1 z_2, z_3\}$ and $\cX = \cD_{5/2}(\bB_3)$. Recall from Example \ref{example:find.max.dom.2} that $\Omega_{max} = \bB_3$. Now, $\fg(\cF) = 1$, $V(\cF)$ is infinite and $V(\cF) \cap \bB_3 = \emptyset$. However, since
    \begin{equation*}
        V(\cF) \cap \bS_3 = \left\{ \left( \frac{t}{\sqrt{2}},\frac{\overline{t}}{\sqrt{2}},0 \right) : t \in \bT \right\}
    \end{equation*}
    has positive \emph{Riesz $5/2$-capacity}, i.e. $\operatorname{cap}_{5/2}(V(\cF) \cap \bS_3) > 0$, the argument from \cite[Section 5]{Vav23} shows that $S[\cF] \neq \cD_{5/2}(\bB_3)$. Essentially, the positive capacity condition allows you to construct a linear functional on $\cD_{5/2}(\bB_3)$ that separates $S[\cF]$ and $1$. Thus, the conclusion of Theorem \ref{thm:joint.cyc.poly} does not hold for $\cF$.
\end{example}

We can also combine Lemma \ref{lemma:fin.codim.ideal.in.P_2} and Theorem \ref{thm:joint.cyc.poly} to get the following description of jointly cyclic polynomials for $d = 2$.

\begin{corollary}\label{cor:joint.cyc.in.2-var}
    If $\cX$ is a Hausdorff TVS of holomorphic functions satisfying \ref{item:P1}-\ref{item:P3} on an open set $\Omega \subset \bC^2$, then a family of polynomials $\cF \subset \cP_2$ is jointly cyclic if and only if $V(\cF) \cap \Omega_{max} = \emptyset$ and $\fg(\cF)$ is cyclic.

    In particular, $V(\cF)$ completely determines the joint cyclicity of $\cF$ if $\fg(\cF) = 1$.
\end{corollary}

\begin{remark}\label{rem:joint.cyc.Dir.type.sp}
    Combining Corollary \ref{cor:joint.cyc.in.2-var} with the main results from \cite{BKKLSS16, KV23} gives a complete description of jointly cyclic polynomials in $H^p(\bD^2)$, $\cD_t(\bD^2)$ and $\cD_t(\bB_2)$.
\end{remark}

\section{Classification of maximal domains}\label{sec:classify.max.dom}

In Examples \ref{example:find.max.dom.1} and \ref{example:find.max.dom.2}, we saw that it is possible for $\Omega_{max}$ to be equal to either $\Omega$ or $\sigma_R(S)$, but it is natural to ask for what $\Gamma \subset \sigma_R(S) \setminus \Omega$ can we find a space $\cX$ with $\Omega_{max} = \Omega \cup \Gamma$.

\subsection{A topological property of maximal domains}\label{subsec:top.prop.max.dom}

In this subsection, we establish a necessary condition on the maximal domain for TVSs $\cX$ that are metrizable, i.e., the topology of $\cX$ is given by an equivalent translation-invariant metric $\operatorname{d}$. In particular, this includes the class of all $F$-spaces and thus, all Banach spaces. Recall that a set is called $F_\sigma$ if it is a countable union of closed sets, and is called $G_\delta$ if it is a countable intersection of open sets.

We showed in Theorem \ref{thm:max.dom.is.point.spec.of.adjoint} that $\Omega_{max} = \sigma_p(S^*)$. Nikol'skaya \cite{Nik74} observed that $\sigma_p(T)$ is always an $F_\sigma$ set for an operator $T$ on a Banach space $\cX$, and the argument presented can be generalized to joint point spectra as well. We present an independent proof of this fact for $\Omega_{max}$ in the general setting of metrizable TVSs.

\begin{theorem} \label{thm: maximal domain is F-sigma}
    If $\cX$ is a metrizable TVS of holomorphic functions satisfying \ref{item:P1}-\ref{item:P3} on an open set $\Omega \subset \bC^d$, then $\Omega_{max}$ is an $F_{\sigma}$ set.
\end{theorem}

\begin{proof}
   Let $\operatorname{d}$ be the translation-invariant metric of $\cX$. For each $n \in \bN$, define
   \begin{equation*}
       U_n := \left\{ F \in \cX : \operatorname{d}(F,0) < 1/2^n \right\}.
   \end{equation*}
   We claim that $w \in \Omega_{max}$ if and only if there exists $N \in \bN$ such that
    \begin{equation} \label{eqn: point in max domain}
        |P(w)| \leq 1 \foral P \in \cP_d \cap U_N.
    \end{equation}
    
    Indeed, if $w \in \Omega_{max}$ then \eqref{eqn: point in max domain} clearly holds for some $N \in \bN$ by the continuity of $\Lambda_w$. Conversely, suppose \eqref{eqn: point in max domain} holds for some $w \in \bC^d$ and $N \in \bN$. It suffices to show that $\Lambda_w$ is well-defined on $\cX$ via limits along polynomials $P \in \cP_d$, i.e.,
    \begin{equation}\label{eqn:eval.at.F.via.polys}
        \Lambda_w(F) := \lim_{P \to F} P(w) \foral F \in \cX,
    \end{equation}
    since this combined with \eqref{eqn: point in max domain} implies the continuity of $\Lambda_w$. To this end, let $F \in \cX$ and $\epsilon > 0$ be fixed but arbitrary, and note that if $P_1, P_2 \in F + \epsilon U_{N+1}$, then $(P_1 - P_2)/\epsilon \in U_N$. \eqref{eqn: point in max domain} then shows that $|P_1(w) - P_2(w)| \leq \epsilon$, which in turn shows that $\Lambda_w(F)$ is well-defined as in \eqref{eqn:eval.at.F.via.polys}. This concludes the claimed equivalence of $w \in \Omega_{max}$ with \eqref{eqn: point in max domain}.
    
    Our result now follows from this equivalence, since
    \begin{equation*}
        \Omega_{max} = \bigcup_{n \in \bN} \left\{ w \in \bC^d: \: |P(w)| \le 1, \ P \in \cP_d \cap U_n \right\}
    \end{equation*}
    is clearly an $F_\sigma$ set.
\end{proof}

\subsection{Maximal domains on the unit disk}\label{subsec:eg.max.dom.unit.disk}

It is now a natural question to ask whether every $F_\sigma$ subset of $\overline{\bD}$ (for instance) appears as a maximal domain of some holomorphic function space on $\bD$. We give a partial answer to this question in this subsection. First, we introduce some important notation.

For two quantities $A,B > 0$ depending on some parameters, we write $A \lesssim B$ (equivalently $B \gtrsim A$) whenever there exists a constant $C > 0$ that is independent of the parameters and such that $A \le CB$. We write $A \asymp B$ if both $A \lesssim B$ and $A \gtrsim B$ hold. Lastly, for a holomorphic function $F$ and any $n \in \bN \cup \{0\}$, we write $F^{(n)}$ to denote the $n^\text{th}$ derivative of $F$. In particular, $F^{(0)} = F$.

Let $dA = \frac{1}{\pi} dxdy$ be the normalized area measure on $\bD$ and let $v \in L^1(\bD)$ be a real-valued function satisfying:
\begin{enumerate}
    \item For every $r < 1$, there exists a $\delta_r > 0$ such that
    \begin{equation}\label{eqn:v.prop.1}
        v(z) > \delta_r \text{ for a.e. } z \in r \bD.
    \end{equation}
    \item For every $r < 1$ we have
    \begin{equation}\label{eqn:v.prop.2}
        v(z) \leq v(rz) \text{ for a.e. } z \in \bD.
    \end{equation}
\end{enumerate}
Lastly, we introduce the \emph{weighted Bergman spaces}
\begin{equation*}
    L^2_a(\bD,v) := \left \{ F \in \operatorname{Hol}(\bD) : \|F\|^2_v := \int_\bD |F|^2 v dA < \infty \right \}.
\end{equation*}

In order to obtain bounded point-evaluations at certain points on $\bT$, we need to control the behavior of the derivatives of functions appropriately. Thus, we introduce the following modification of $L^2_a(\bD,v)$.

\begin{definition}
    For any $n \in \bZ_+$ and $v \in L^1(\bD)$ satisfying \eqref{eqn:v.prop.1} and \eqref{eqn:v.prop.2}, define
    \begin{equation*}
         \cH_{v, n} = \left\{F \in \operatorname{Hol}(\bD): \: \|F\|^2_{v, n} := \sum_{k = 0}^n \| F^{(k)} \|^2_v < \infty \right\}.
    \end{equation*}
\end{definition}

\begin{theorem} \label{thm: v n spaces}
    $\cH_{v, n}$ is a Hilbert space satisfying \ref{item:P1}-\ref{item:P3} on $\bD$ with $\Omega_{max} \subseteq \overline{\bD}$.
\end{theorem}
\begin{proof}
    Fix $r < 1$.
    Since $|F^{(k)}|^2$ is subharmonic for each $0 \leq k \leq n$, we get that
    \begin{equation} \label{eqn:H_v,n.has.P2}
        |F^{(k)}(z)|^2 \le \frac{4}{(1-r)^2} \int_{D\left(z,\frac{1-r}{2}\right)} |F^{(k)}|^2 dA \lesssim \frac{\delta^{-1}_{\frac{1+r}{2}}}{(1-r)^2} \|F\|^2_{v,n} \foral z \in r \bD,
    \end{equation}
    where we use \eqref{eqn:v.prop.1} for the second inequality. 
    Hence, by the standard normal families argument together with Fatou's lemma, we get that $\cH_{v, n}$ is indeed a Hilbert space. Also, \eqref{eqn:H_v,n.has.P2} for $k = 0$ shows that $\cH_{v,n}$ satisfies \ref{item:P2}. It remains to check that \ref{item:P1} and \ref{item:P3} hold, and that $\Omega_{max} \subseteq \overline{\bD}$.

    First, we check \ref{item:P3}. It is clear from the definition of $\left\|\cdot\right\|_{v,n}$ that
    \begin{equation}\label{eqn:norm.asymp}
        \|F\|_{v,n} \asymp \sum_{k = 0}^n \| F^{(k)} \|_v.
    \end{equation}
    Also, one can inductively check that
    \begin{equation}\label{eqn:induc.on.SF}
        \| (SF)^{(k)} \|_v \lesssim \|F\|_{v,n} \foral 0 \leq k \leq n.
    \end{equation}
    Combining \eqref{eqn:norm.asymp} with \eqref{eqn:induc.on.SF} gives us
    \begin{equation*}
        \|SF\|_{v,n} \lesssim \|F\|_{v,n} \foral F \in \cH_{v,n}
    \end{equation*}
    and thus, $\cH_{v,n}$ satisfies \ref{item:P3}. Next, we check \ref{item:P1}. Note that
    \begin{equation*}
        \|z^k\|_{v,n} \lesssim (1 + k)^{n} \foral k \in \bZ_+,
    \end{equation*}
    which implies that if $F \in \operatorname{Hol(R \bD)}$ for some $R > 1$, then its Taylor series converges absolutely in $\cH_{v,n}$. In particular, $F \in \cH_{v,n}$. It also implies that $\cP_1 \subset \cH_{v,n}$ and functions in $\operatorname{Hol}(R\bD)$ for $R > 1$ can be approximated in $\cH_{v,n}$ by polynomials. Thus, it is enough to show that any $F \in \cH_{v,n}$ can be approximated by its \emph{dilations}
    \begin{equation*}
        F_r : z \mapsto F(rz) \foral r < 1.
    \end{equation*}
    For all $0 \leq k \leq n$ and $r < 1$, we have
    \begin{equation*}
        (F_r)^{(k)}(z) = r^k F^{(k)}(rz) = r^k (F^{(k)})_r(z) \foral z \in \bD.
    \end{equation*}
    Then, using \eqref{eqn:norm.asymp}, it suffices to show that dilations converge in the norm for $L^2_a(\bD,v)$. This is a standard fact for weighted Bergman spaces, e.g., see \cite[Proposition 1.3]{HKZ00} and use \eqref{eqn:v.prop.2} above for $v$ instead of the weight they use.

    Lastly, suppose $w \in \Omega_{max} \setminus \overline{\bD}$. Note that $F(z) = 1/(z - w) \in \cH_{v,n}$ since it is holomorphic on a larger disk. We saw that this implies $F_r \to F$ in $\cH_{v,n}$ and that $F_r$ can be approximated by its Taylor series in $\cH$. In particular,
    \begin{equation*}
        \Lambda_w(F_r) = F(rw) \foral r < 1
    \end{equation*}
    and we get
    \begin{equation}\label{eqn:eval.at.F.in.terms.of.dilations}
        \Lambda_w(F) = \lim_{r \to 1} F(rw),
    \end{equation}
    which clearly diverges. Therefore, $\Omega_{max} \subseteq \overline{\bD}$ and the proof is complete.
\end{proof}

Let us use this general construction to find function spaces with $\Omega_{max} = \bD \cup \Gamma$ for some $\Gamma \subset \bT$ by considering appropriate weights $v$. The goal is to make $v$ vanish quickly around $w \in \bT \setminus \Gamma$ and bounded below in some neighborhood of $w \in \Gamma$.

\begin{theorem} \label{thm: closed Gamma}
    Suppose $\Gamma \subset \bT$ is a closed set. Then, there exists a Hilbert space $\cH$ satisfying \ref{item:P1}-\ref{item:P3} on $\bD$ such that $\Omega_{max} = \bD \cup \Gamma$.
\end{theorem}

\begin{proof}
    We consider $\cH = \cH_{v, 2}$ with $v(z) = \exp\left(- \frac{\dist(z/|z|, \Gamma)}{1 - |z|}\right)$. Note that $v$ satisfies \eqref{eqn:v.prop.1} and \eqref{eqn:v.prop.2}, so we only need to check that $\Gamma = \Omega_{max} \cap \bT$.

    First, suppose $\gamma \in \Gamma$. We need to show $\gamma \in \Omega_{max}$. Fix an angle $0 < \alpha < \pi/2$ and a radius $\rho < 1/2$, and consider the \emph{Stolz angle} $S_{\alpha,\rho}$ at $\gamma$, i.e,
    \begin{equation*}
        S_{\alpha,\rho} = \left\{ \gamma - r e^{i \theta} : 0 < r < \rho \AND |\theta| < \alpha \right\}.
    \end{equation*}
    For $z \in S_{\alpha,\rho}$,
    \begin{equation*}
        \dist\left(\frac{z}{|z|}, \Gamma\right) \le \left\lvert\frac{z}{|z|} - \gamma\right\rvert \lesssim 1 - |z|.
    \end{equation*}
    Hence, $v \gtrsim 1$ on $S_{\alpha,\rho}$. By $[0,w)$ we mean a line segment joining $0$ and $w$. For $z \in [0, \gamma)$, consider the disc $D(z, r_z) \subset S_{\alpha,\rho}$ with the largest possible radius $r_z$. For any $P \in \cP_1$ and $z \in [0,\gamma)$, the subharmonicity of $|P^{(2)}|^2$ shows that
    \begin{equation*}
        |P^{(2)}(z)|^2 \le \frac{1}{r_z^2} \int_{D(z, r_z)} |P^{(2)}|^2 dA. 
    \end{equation*}
    Then, since $v \gtrsim 1$ on $S_{\alpha,\rho}$ and $r_z \asymp 1 - |z|$, we get
    \begin{equation*}
        |P^{(2)}(z)| \lesssim \frac{1}{1 - |z|} \|P\|_{v,2}.
    \end{equation*}
    Applying the fundamental theorem of calculus to $P^{(2)}$ on $[0,z)$ and using \eqref{eqn:H_v,n.has.P2} with $k = 1$ and $z = 0$, we obtain
    \begin{equation*}
        |P^{(1)}(z)| \lesssim (1 - \log(1 - |z|)) \|P\|_{v,2} \foral z \in [0,\gamma).
    \end{equation*}
    Doing this once more with $P^{(1)}$ and noting that $\log (1 - t) \in L^1([0,1))$, we get
    \begin{equation*}
        |P(\gamma)| \lesssim \|P\|_{v,2}.
    \end{equation*}
    Since $P \in \cP_1$ was arbitrarily chosen, it follows that $\gamma \in \Omega_{max}$.

    All we need to show now is that if $w \in \bT \setminus \Gamma$, then $w \notin \Omega_{max}$. By the convergence of dilations (see the discussion surrounding \eqref{eqn:eval.at.F.in.terms.of.dilations}), it is enough to conclude that $F(z) = 1/(z - w) \in \cH$ for all $w \in \bT \setminus \Gamma$. Note that
    \begin{equation*}
        \|F\|_{v,2}^2 \lesssim \left\| \frac{1}{(z-w)^3} \right\|_v^2.
    \end{equation*}
    Hence, it is sufficient to show that the RHS above is finite.
    Since $w \notin \Gamma$ and $\Gamma$ is closed, we have $\dist(w, \Gamma) = d > 0$. We consider two regions
    \begin{equation*}
        D_1 := \{ z \in \bD : |z - w| < d/2 \} \AND D_2 := \bD \setminus D_1,
    \end{equation*}
    and estimate $\| 1/(z-w)^3 \|_v^2$ separately over $D_1$ and $D_2$. First, note that
    \begin{equation}\label{eqn:approx.D_2}
       \int_{D_2} \frac{1}{|z - w|^6} v(z) dA(z) \lesssim \frac{1}{d^6} \int_{D_2} v(z) dA(z) \lesssim \frac{1}{d^6}.  
    \end{equation}
    Now, $z \in D_1$ and $\dist(w, \Gamma) = d$ implies $\dist(z, \Gamma) \ge d/2$. Thus,
    \begin{align*}
        \int_{D_1} \frac{1}{|z - w|^6} v(z) dA(z) &\leq \int_{D_1} \frac{1}{(1 - |z|)^6} e^{-\frac{d}{2(1 - |z|)}} dA(z) \\
        &\leq \int_{\bD} \frac{1}{(1 - |z|)^6} e^{-\frac{d}{2(1 - |z|)}} dA(z) \\
        &\leq \int_0^1 \frac{1}{(1-r)^6} e^{- \frac{d}{2(1-r)}} dr \lesssim \frac{1}{d^6}.
    \end{align*}
    Combining this with \eqref{eqn:approx.D_2}, we conclude the proof since
    \begin{equation}\label{eqn:estimate.1/(z-w)}
        \left\| \frac{1}{(z-w)^3} \right\|_v^2 \lesssim \frac{1}{(\dist(w, \Gamma))^6} < \infty. \qedhere
    \end{equation}
    \end{proof}

Let us show how this construction extends to $\Gamma \subset \bT$ that is both $F_\sigma$ and $G_\delta$.

\begin{theorem} \label{thm: Fs and Gd Gamma}
    Suppose $\Gamma \subset \bT$ is both an $F_\sigma$ set and a $G_\delta$ set. Then, there exists a Hilbert space $\cH$ satisfying \ref{item:P1}-\ref{item:P3} on $\bD$ such that $\Omega_{max} = \bD \cup \Gamma$.
\end{theorem}
\begin{proof}
    Since $\Gamma$ is $F_{\sigma}$, we have $\Gamma = \cup_k \Gamma_k$ for some closed $\Gamma_k \subset \bT$. For each $\Gamma_k$, let $v_k$ be the weight from Theorem \ref{thm: closed Gamma}. Now, define $\cH$ to be $\cH_{v, 2}$ with $v = \sum_{k} a_k v_k$ for some $a_k > 0$ with $\sum_k a_k < \infty$, which will be specified later. It follows that $v$ satisfies \eqref{eqn:v.prop.1} and \eqref{eqn:v.prop.2}, so we only need to show that $\Omega_{max} \cap \bT = \Gamma$.

    First, note that
    \begin{equation*}
        \|F\|^2_{v_k,2} \leq \frac{1}{a_k} \|F\|^2_{v,2} \foral k \in \bZ_+.
    \end{equation*}
    In light of Theorem \ref{thm: closed Gamma}, this means that $\cH \subset \cH_{v_k,2}$ boundedly and thus, each $\Gamma_k \subset \Omega_{max}$. Since this is true for every $k \in \bZ_+$, we conclude that $\Gamma \subset \Omega_{max}$.

    Now, we only need to choose $a_n$'s such that $w \in \bT \setminus \Gamma$ implies $w \notin \Omega_{max}$. $\Gamma$ being $G_{\delta}$ means that $\bT \setminus \Gamma = \cup_l C_l$, for some closed $C_l \subset \bT$. Since $\Gamma_k$'s and $C_l$'s do not intersect, we have that
    \begin{equation*}
        \dist(\Gamma_k, C_l) =: d_{k,l} > 0 \foral k,l \in \bZ_+.
    \end{equation*}
    As in Theorem \ref{thm: closed Gamma}, it suffices to show that $F(z) = 1/(z - w) \in \cH$ for each $w \in \bT \setminus \Gamma$. From \eqref{eqn:estimate.1/(z-w)}, we know that
   \begin{equation*}
        \|F\|_{v,2}^2 \lesssim \sum_{k \in \bZ_+} \frac{a_k}{(\dist(w, \Gamma_k))^6}.
    \end{equation*}
    Thus, it suffices to have
    \begin{equation}\label{eqn:condition.on.a_k}
        \sum_{k \in \bZ_+} \frac{a_k}{d_{k,l}^6} < \infty \foral l \in \bZ_+.
    \end{equation}
    Let us now specify the $a_k$'s that achieve this. Set
    \begin{equation*}
        a_k = \frac{1}{2^k} \prod_{j = 0}^{k} \left( \frac{d_{k,j}}{2} \right)^6 \foral k \in \bZ_+.
    \end{equation*}
    Clearly, $a_k > 0$ and $\sum_k a_k < \infty$ since $d_{k,l} \leq 2$ for each $k,l \in \bZ_+$. Finally, fixing $l \in \bZ_+$, we can estimate for $k \ge l$ that
    \begin{equation*}
        \frac{a_k}{d_{k,l}^6} \le \frac{1}{2^k} \left( \frac{d_{k,l}}{2} \right)^6 \frac{1}{d_{k,l}^6} \le \frac{1}{2^k}.
    \end{equation*}
    Thus, \eqref{eqn:condition.on.a_k} holds, and this completes the proof.
\end{proof}

\begin{remark}
    Examples of $\Gamma \subset \bT \cong [0,1)$ for which the above theorem works are: $(1)$ any closed/open set, e.g., finite collection of closed/open sub-arcs of $\bT$, the Cantor set and its complement, etc. $(2)$ any countable union of intervals with countably many limit points in the complement, e.g., any sequence with countably many limit points, etc. Since our argument crucially relies on $\Gamma$ also being $G_\delta$, we do not know if our construction generalizes to, for instance, $\Gamma = \bQ \cap [0,1)$. We therefore ask the following question.
\end{remark}

\begin{problem}\label{prob:dichotomy.max.dom}
    Does there exist a Hilbert/Banach space $\cH$ satisfying \ref{item:P1}-\ref{item:P3} on $\bD$ with $\Omega_{max} = \bD \cup \Gamma$ for any given $F_\sigma$ set $\Gamma \subset \bT$?
\end{problem}

\begin{remark}\label{rem:loc.con.sp.with.F_sigma.max.dom}
    The above problem can be solved for locally convex TVSs using a similar idea as the previous theorems. We provide a sketch of the argument here. Suppose $\Gamma = \cup_{k} \Gamma_k$ with $\Gamma_k$ closed, and consider the seminorms for $n,k \in \bN$
    \begin{equation*}
        p_n(f) = \sup_{z \in (\frac{n-1}{n})\bD} |f(z)| \AND q_k(f) = \sup_{z \in \Gamma_k} |f(z)|.
    \end{equation*}
    We now define $\cX$ as the locally convex TVS that is obtained by completing $\cP_1$ with these seminorms. 
    It is routine to check that $\cX$ satisfies \ref{item:P1}-\ref{item:P3} on $\bD$, and that $\bD \cup \Gamma \subseteq \Omega_{max}$. It remains to show that $\Omega_{max} \subseteq \bD \cup \Gamma$.
    
    Let $w \notin \bD \cup \Gamma$ and note that $f_r(z) = 1/(rz - w) \in \cX$ for each $r < 1$ using the absolute convergence of its Taylor series in each seminorm. Now, if $\Lambda_w$ extends continuously to $\cX$, then it is bounded by a finite linear combination of the seminorms. It remains to notice that $p_n(f_r)$ and $q_k(f_r)$ are bounded as $r \to 1^-$, while $|\Lambda_w(f_r)| \to \infty$, which gives us a contradiction. This means that $\Lambda_w$ is not continuous, that is, $w \notin \Omega_{max}$.
\end{remark}

\subsection*{Acknowledgements}

We would like to thank Alberto Dayan, Ramlal Debnath, Greg Knese, James Pascoe and Dimitrios Vavitsas for several helpful discussions. We also extend our gratitude to our advisor Orr Moshe Shalit for his support on this project. Lastly, we thank the anonymous referee for their valuable input in improving the content of this paper.

\printbibliography

@article{AC70,
  title={Invariant subspaces and analytic continuation in several variables},
  author={P. Ahern and D. Clark},
  journal={J. Math. Mech.},
  volume={\textbf{19}},
  number={11},
  pages={963--969},
  year={1970},
  publisher={JSTOR}
}

@article{APRSS24,
  title={Cyclicity in the Drury-Arveson space and other weighted Besov spaces},
  author={A. Aleman and K.M. Perfekt and S. Richter and C. Sundberg and J. Sunkes},
  journal={Trans. Amer. Math. Soc.},
  volume={\textbf{377}},
  number={02},
  pages={1273--1298},
  year={2024}
}

@book{AM02,
  title={Pick interpolation and Hilbert function spaces},
  author={J. Agler and J.E. McCarthy},
  volume={\textbf{44}},
  year={2002},
  publisher={Amer. Math. Soc.}
}

@article{BCLSS15,
  title={Cyclicity in Dirichlet-type spaces and extremal polynomials II: functions on the bidisk},
  author={C. B{\'e}n{\'e}teau and A. Condori and C. Liaw and D. Seco and A. Sola},
  journal={Pacific J. Math.},
  volume={\textbf{276}},
  number={1},
  pages={35--58},
  year={2015},
  publisher={Mathematical Sciences Publishers}
}

@article{BKKLSS16,
  title={Cyclic polynomials in two variables},
  author={C. B{\'e}n{\'e}teau and G. Knese and {\L}. Kosi{\'n}ski and C. Liaw and D. Seco and A. Sola},
  journal={Trans. Amer. Math. Soc.},
  volume={\textbf{368}},
  number={12},
  pages={8737--8754},
  year={2016}
}

@article{Beu49,
title="On two problems concerning linear transformations in Hilbert space",
author="A. Beurling",
journal="Acta Math.",
volume="\textbf{81}",
pages="239--255",
year="1949"
}

@article{BH97,
  title={Harmonic functions of maximal growth: {I}nvertibility and cyclicity in {B}ergman spaces},
  author={A. Borichev and H. Hedenmalm},
  journal={J. Amer. Math. Soc.},
  volume={\textbf{10}},
  pages={761--796},
  year={1997}
}

@article{BS84,
  title="Cyclic vectors in the Dirichlet space",
  author="L. Brown and A. Shields",
  journal="Trans. Amer. Math. Soc.",
  volume="\textbf{285}",
  number="1",
  pages="269--303",
  year="1984"
}

@book{CG03,
  title={Analytic Hilbert modules},
  author={X. Chen and K. Guo},
  year={2003},
  publisher={Chapman \& Hall}
}

@book{CMR20,
  title={Function theory and $\ell^p$ spaces},
  author={R. Cheng and J. Mashreghi and W. Ross},
  volume={\textbf{75}},
  year={2020},
  publisher={Amer. Math. Soc.}
}

@article{DP89,
  title={Hilbert modules over function algebras},
  author={R. Douglas and V. Paulsen},
  journal={Pitman Res. Notes Math.},
  volume={\textbf{217}},
  year={1989}
}

@book{Dur70,
  title="Theory of $H^p$ spaces",
  author="P.L. Duren",
  year="1970",
  publisher="Academic Press"
}

@book{EKMR14,
  title={A primer on the Dirichlet space},
  author={O. El-Fallah and K. Kellay and J. Mashreghi and T. Ransford},
  volume={\textbf{203}},
  year={2014},
  publisher={Cambridge University Press}
}

@article{Gam66,
  title={{$H^p$} {S}paces and {E}xtremal {F}unctions in {$H^1$}},
  author={T. Gamelin},
  journal={Trans. Amer. Math. Soc.},
  number={1},
  pages={158--167},
  publisher={American Mathematical Society},
  volume={\textbf{124}},
  year={1966}
}

@article{Guo99,
  title={Characteristic spaces and rigidity for analytic Hilbert modules},
  author={K. Guo},
  journal={J. Funct. Anal.},
  volume={\textbf{163}},
  number={1},
  pages={133--151},
  year={1999},
  publisher={Elsevier}
}

@article{Har17,
  title={On the isomorphism problem for multiplier algebras of Nevanlinna-Pick spaces},
  author={M. Hartz},
  journal={Canad. J. Math.},
  volume={\textbf{69}},
  number={1},
  pages={54--106},
  year={2017},
  publisher={Cambridge University Press}
}

@book{HKZ00,
  title={Theory of Bergman spaces},
  author={H. Hedenmalm and B. Korenblum and K. Zhu},
  year={2000},
  volume={\textbf{199}},
  publisher={Springer}
}

@article{KLRS19,
  title={Cyclic polynomials in anisotropic Dirichlet spaces},
  author={G. Knese and {\L}. Kosi{\'n}ski and T. Ransford and A. Sola},
  journal={J. Anal. Math.},
  volume={\textbf{138}},
  pages={23--47},
  year={2019},
  publisher={Springer}
}

@book{KR00,
  title={Computational commutative algebra Vol. \textbf{1}},
  author={M. Kreuzer and L. Robbiano},
  year={2000},
  publisher={Springer}
}

@article{KV23,
  title={Cyclic polynomials in Dirichlet-type spaces in the unit ball of $\bC^2$},
  author={{\L}. Kosi{\'n}ski and D. Vavitsas},
  journal={Constr. Approx.},
  volume={\textbf{58}},
  number={2},
  pages={343--361},
  year={2023},
  publisher={Springer}
}

@article{MS17,
  title={Spaces of Dirichlet series with the complete Pick property},
  author={J.E. McCarthy and O.M. Shalit},
  journal={Israel J. Math.},
  volume={\textbf{220}},
  pages={509--530},
  year={2017},
  publisher={Springer}
}

@article{NGN70,
  title={Approximation by $\{f(kx)\}$},
  author={J.H. Neuwirth and J. Ginsberg and D.J. Newman},
  journal={J. Funct. Anal.},
  volume={\textbf{5}},
  number={2},
  pages={194--203},
  year={1970},
  publisher={Academic Press}
}

@article{Nik74,
  title={Structure of the point spectrum of a linear operator},
  author={L.N. Nikol'skaya},
  journal={Math. Notes, USSR},
  volume={\textbf{15}},
  number={1},
  pages={83--87},
  year={1974},
  publisher={Springer}
}

@article{Nik12,
  title={In a shadow of the RH: cyclic vectors of Hardy spaces on the Hilbert multidisc},
  author={N. Nikolski},
  journal={Ann. Inst. Fourier},
  volume={\textbf{62}},
  number={5},
  pages={1601--1626},
  year={2012}
}

@book{Rud69,
  title="Function theory in polydiscs",
  author="W. Rudin",
  year="1969",
  volume="\textbf{41}",
  publisher="W. A. Benjamin"
}

@book{Rud73,
  title={Functional analysis},
  author={W. Rudin},
  year={1973},
  publisher={McGraw-Hill}
}

@article{Sam21,
  title={Cyclicity preserving operators on spaces of analytic functions in $\bC^n$},
  author={J. Sampat},
  journal={Integr. Equ. Oper. Theory},
  volume={\textbf{93}},
  pages={1--20},
  year={2021},
  publisher={Springer}
}

@article{ST21,
  title={Polynomial approach to cyclicity for weighted $\ell^p_A$},
  author={D. Seco and R. T{\'e}llez},
  journal={Banach J. Math. Anal.},
  volume={\textbf{15}},
  pages={1--16},
  year={2021},
  publisher={Springer}
}

@article{Vav23,
  title={A note on cyclic vectors in Dirichlet-type spaces in the unit ball of},
  author={D. Vavitsas},
  journal={Canad. Math. Bull.},
  volume={\textbf{66}},
  number={3},
  pages={886--902},
  year={2023},
  publisher={Canadian Mathematical Society}
}

@article{VZ24,
  title={Non-cyclicity and polynomials in Dirichlet-type spaces of the unit ball},
  author={D. Vavitsas and K. Zarvalis},
  journal={Bull. London Math. Soc.},
  volume={\textbf{56}},
  number={12},
  pages={3905--3919},
  year={2024},
  publisher={Wiley Online Library}
}

@article{Yan99,
  title={The Berger-Shaw theorem in the Hardy module over the bidisk},
  author={R. Yang},
  journal={J. Operator Theory},
  volume={\textbf{42}},
  pages={379--404},
  year={1999},
  publisher={JSTOR}
}

\end{document}